\documentclass[a4paper,11pt,reqno]{amsart}
\numberwithin{equation}{section}
\usepackage{amsmath}
\usepackage{amssymb}

\usepackage{xcolor}

\theoremstyle{plain} 
\newtheorem{theorem}{Theorem}[section]
\newtheorem{proposition}[theorem]{Proposition}
\newtheorem{lemma}[theorem]{Lemma}

\theoremstyle{remark}
\newtheorem{remark}[theorem]{Remark}

\title[Linearly edge-reinforced random walks on $\mathbb{Z}_+$]{Almost sure behavior of linearly edge-reinforced random walks on the half-line}
\author[M. Takei]{Masato Takei}
\address{Department of Applied Mathematics, Faculty of Engineering, Yokohama National University, Hodogaya, Yokohama 240-8501, Japan.}
\email{takei-masato-fx@ynu.ac.jp}
\thanks{M.T. is partially supported by JSPS Grant-in-Aid for Young Scientists (B) No. 16K21039, and JSPS Grant-in-Aid for Scientific Research (B) No. 19H01793 and (C) No. 19K03514.} 

\begin{document}
\dedicatory{To the memory of late Professor Munemi Miyamoto}


\begin{abstract}
We study linearly edge-reinforced random walks on $\mathbb{Z}_+$, where each edge $\{x,x+1\}$ has the initial weight $x^{\alpha} \vee 1$, and each time an edge is traversed, its weight is increased by $\Delta$. It is known that the walk is recurrent if and only if $\alpha \leq 1$.
The aim of this paper is to study the almost sure behavior of the walk in the recurrent regime. For $\alpha<1$ and $\Delta>0$, we obtain a limit theorem which is a counterpart of the law of the iterated logarithm for simple random walks.
This reveals that the speed of the walk with $\Delta>0$ is much slower than $\Delta=0$. In the critical case $\alpha=1$, our (almost sure) bounds for the trajectory of the walk shows that there is a phase transition of the speed at $\Delta=2$.
\end{abstract}

\maketitle

\section{Introduction}

Reinforced random walks (RRWs), introduced by Coppersmith and Diaconis, are a class of self-interacting random walks that have attracted many researchers for three decades or more. Quoting from Diaconis \cite{Diaconis88}, 
\begin{quote}
{\it ``It was introduced as a simple model of exploring a new city. At first all routes are equally unfamiliar, and one chooses at random between them. As time goes on, routes that have been traveled more in the past are more likely to be traveled."}
\end{quote}
Consider a finite connected graph, and each edge is given a positive initial weights.
In each step the traveller jumps to an adjacent vertex by traversing an edge, with probability proportional to the weight of that edge. Each time an edge is traversed, its weight is increased by a fixed constant $\Delta>0$ {\it (linear edge-reinforcement)}. We can see that the walk is recurrent, that is, every vertex is visited infinitely often with probability one. The limiting density of the normalized occupation measure on the edges, obtained by Coppersmith and Diaconis in 1986, is found in \cite{Diaconis88} (see also Keane and Rolles \cite{KeaneRolles00}).

In this paper we consider the linearly edge-reinforced random walk (LERRW) on the half-line in recurrent regime, and give almost sure results on how far the traveller is from the origin. We begin with a motivating example. Let $\{S_n\}$ be the symmetric simple random walk on $\mathbb{Z}$, starting at the origin. Notice that $\{|S_n|\}$ is the symmetric simple random walk on $\mathbb{Z}_+=\{0,1,2,\cdots\}$, with a reflection at the origin. This walk is recurrent a.s., that is,
\[ \liminf_{n \to \infty} |S_n|=0, \quad \mbox{and} \quad \limsup_{n \to \infty} |S_n|=+\infty \quad \mbox{a.s..} \]
More precisely, by the law of the iterated logarithm, we have
\begin{align*}
\limsup_{n \to \infty} \dfrac{|S_n|}{\sqrt{2n\log \log n}} = 1 \quad \mbox{a.s..}
\end{align*}
Now consider the LERRW on $\mathbb{Z}_+$,
where the initial weights are all one and $\Delta=1$. Let $X_n$ be the position of the walk at time $n$, and assume that $X_0=0$.
Again $\{X_n\}$ is recurrent a.s. (see \cite{Davis90}), but is quantitatively quite different from $\{|S_n|\}$: Theorem \ref{thm:Takei20Main1} below
implies that 
\begin{align*}
\limsup_{n \to \infty} \dfrac{X_n}{\log_4 n} = 1 \quad \mbox{a.s..}
\end{align*}
In this way, the random walk with reinforcement is much slower than ordinary random walk.

We also discuss how strongly the linear reinforcement affects the long time behavior of the walk on $\mathbb{Z}_+$ with heterogeneous initial weights, where each edge $\{x,x+1\}$ has initial weight $x^{\alpha} \vee 1$. To summarize our aim is to describe phase transitions of the speed in the recurrent regime $\alpha \leq 1$ and $\Delta \geq 0$. 
For $\alpha<1$, we obtain a limit theorem which shows strong slow-down effects from $\Delta=0$, and is a counterpart of the law of the iterated logarithm for simple random walks.  To our best knowledge this kind of results are first for RRWs. On the other hand, for $\alpha=1$, which is critical for recurrence, essential slow-down effects appear only for $\Delta>2$.

\section{Definitions and Results} \label{sec:DefResults}

\subsection{Model} We define the {\it edge-reinforced random walk} (ERRW) on $\mathbb{Z}_+$, denoted by $\boldsymbol{X} =\{X_n\}$, as follows. This process takes values on the vertices of $\mathbb{Z}_+=\{0,1,2,\ldots\}$, and at each step it jumps to one of the nearest neighbors. 
For each $x \in \mathbb{Z}_+$, let $\mathbf{f}_{x} =(f(\ell,x) \colon \ell \in \mathbb{Z}_+)$ be a non-decreasing sequence of positive numbers, called the {\it reinforcement scheme} at $x$. 
For $x \in \mathbb{Z}_+$, let $\phi_n(x)$ be the number of traversals of the edge $\{x,\,x+1\}$ by time $n$, namely
\begin{equation}
\label{def:phi}
 \phi_n(x) := \sum_{i=1}^n 1_{ \{ \{X_{i-1},\,X_i\} = \{x,\,x+1\} \}}. 
\end{equation}
For each $n \geq 0$, the weights at time $n$ are defined by  
\[
w_n (x) = f(\phi_n(x),x)\quad \mbox{for $x \in \mathbb{Z}_+$}, 
\]
We set $w_n(-1) = 0$ for all $n$, which implies a reflection at the origin.
We call $w_0(x)=f(0,x)$ the {\it initial weight} of the edge $\{x,x+1\}$.
Assume that $P(X_0=0)=1$. The transition probability is given by
\begin{align}
P(X_{n+1}=X_n+1\,|\,X_0,\ldots,X_n)&=1-P(X_{n+1}=X_n-1\,|\,X_0,\ldots,X_n) \notag \\
 &=\frac{w_n (X_n)}{w_n(X_n-1)+w_n(X_n)}. \label{eq:ERRWtransition}
\end{align}

The {\it linearly edge-reinforced random walk} (LERRW) is the ERRW whose reinforcement scheme is defined by
\begin{align}
 f(\ell,x)=f(0,x) + \ell \Delta \quad \mbox{for $\ell,x \in \mathbb{Z}_+$}. \label{eq:LERRWschemeDef}
\end{align}
We call $\Delta \geq 0$ the {\it reinforcement parameter}.

\subsection{Recurrence classification}
We say that the path $\boldsymbol{X}$ is {\it recurrent} if every point is visited infinitely often, and {\it transient} if every point is visited only finitely many times. 
The recurrence problem for the LEERW on $\mathbb{Z}_+$ is solved by Takeshima \cite{Takeshima00} (although only the $\Delta=1$ case is treated in \cite{Takeshima00}, his argument works for any $\Delta > 0$ as well). In Appendix \ref{sec:AppendixRecTraPres}, we give an elementary and short proof of Theorem \ref{thm:Takeshima00linear}.

\begin{theorem}[\cite{Takeshima00}, Theorem 4.1] \label{thm:Takeshima00linear}
Let $\boldsymbol{X}$ be the linearly edge-reinforced random walk on $\mathbb{Z}_+$
with the reinforcement parameter $\Delta \geq 0$. Define
\[ F_0:= \sum_{x=0}^{\infty} \frac{1}{f(0,x)}. \]
\begin{itemize}
\item[(i)] If $F_0=+\infty$, then $\boldsymbol{X}$ is recurrent a.s..
\item[(ii)] If $F_0<+\infty$, then $\boldsymbol{X}$ is transient a.s..
\end{itemize}
\end{theorem}

Theorem \ref{thm:Takeshima00linear} shows that the recurrence of the LERRW is completely determined by the initial weights: In particular, if the walk is transient when $\Delta=0$, then it never becomes recurrent even if $\Delta>0$ is very large. 

\subsection{Main results: Almost sure behavior of LERRWs} Consider the LERRW $\boldsymbol{X}$ on $\mathbb{Z}_+$ with the initial weight
\begin{align}
f(0,x) = x^{\alpha} \vee 1 =
\begin{cases}
 1 & \mbox{($x=0$),} \\
 x^{\alpha} & \mbox{($x \in \mathbb{N}=\{1,2,\cdots\}$),} \\
 \end{cases}
 \label{eq:InitWeightPowerIkenami}
\end{align}
and the reinforcement parameter $\Delta>0$.
By Theorem \ref{thm:Takeshima00linear}, $\boldsymbol{X}$ is recurrent a.s. if and only if $\alpha \leq 1$. Ikenami (Master thesis) \cite{Ikenami01} shows that if $0 \leq \alpha < 1$ and the reinforcement parameter $\Delta=1$, then for any $\varepsilon>0$,
\[
 \lim_{n \to \infty} \frac{X_n}{(\log n)^{(1+\varepsilon)/(1-\alpha)}} =0
\quad \mbox{a.s..}
\]
The proof in \cite{Ikenami01} is inspired by the Lyapunov function method in Comets, Menshikov, and Popov \cite{CometsMenshikovPopov98}.

Our first result is for the off-critical case, $\alpha<1$, and the precise order of oscillation of $X_n$ is indeed $(\log n)^{1/(1-\alpha)}$.

\begin{theorem} \label{thm:Takei20Main1} Assume that $\alpha <1$ and $\Delta >0$. Let
\begin{align}
K(\alpha,\Delta) := 
\begin{cases}
\dfrac{1-\alpha}{2\Delta} &(\alpha<0), \\[2mm]
\left( \Psi\left(\dfrac{1}{2\Delta}+\dfrac{1}{2}\right)-\Psi\left(\dfrac{1}{2}\right) \right)^{-1}&(\alpha=0), \\[3mm]
\dfrac{1-\alpha}{\Delta} &(0<\alpha<1), \\
\end{cases}
\label{eq:Takei20Main1Const}
\end{align}
where $\Psi(z) = \Gamma'(z) / \Gamma(z)$ is the digamma function.
The LERRW $\boldsymbol{X}$ with the initial weight \eqref{eq:InitWeightPowerIkenami} and the reinforcement parameter $\Delta$ satisfies that
\begin{align*}
\limsup_{n \to \infty} \dfrac{X_n}{\{ K(\alpha,\Delta)\log n \}^{1/(1-\alpha)}} = 1
\quad \mbox{a.s..}
\end{align*}
\end{theorem}

\begin{remark}
It is known that $K(0,1)=1/(\log 4)$, and $K(0,\Delta) \sim 1/(2\Delta)$ as $\Delta \to \infty$.
(See \eqref{eq:digammaDifflog4} and Lemma \ref{lem:Takeshima00(4.10)(4.12)improved} below, respectively.)
\end{remark}

The next theorem is concerning the growth in the critical case, $\alpha=1$.

\begin{theorem} \label{thm:Takei20Main2} Assume that $\alpha =1$ and $\Delta >0$, and consider the LERRW $\boldsymbol{X}$ with the initial weight \eqref{eq:InitWeightPowerIkenami} and the reinforcement parameter $\Delta$. 
\begin{itemize}
\item[(i)] If $\Delta >2$, then for any $\varepsilon>0$,
\begin{align*}
\limsup_{n \to \infty} \dfrac{X_n}{n^{(1-\varepsilon)/\Delta}} =+\infty, 
\quad \mbox{and} \quad 
\lim_{n \to \infty} \dfrac{X_n}{n^{(1+\varepsilon)/\Delta}} = 0\quad \mbox{a.s..}
\end{align*}
\item[(ii)] If $0< \Delta \leq 2$, then for any $\varepsilon>0$,
\begin{align*}
\limsup_{n \to \infty} \dfrac{X_n}{n^{(1-\varepsilon)/2}} =+\infty, 
\quad \mbox{and} \quad 
\lim_{n \to \infty} \dfrac{X_n}{n^{(1+\varepsilon)/2}} = 0
\quad \mbox{a.s..}
\end{align*}
\end{itemize}
\end{theorem}

\subsection{Effect of linear reinforcement} For comparison, we give almost sure bounds for unreinforced case. In the case $\alpha<1$ the speed of the walker becomes much slower as soon as $\Delta>0$, while it is not in the critical case $\alpha=1$.

\begin{theorem} \label{thm:Takei20Main3}
Consider the LERRW $\boldsymbol{X}$ with the initial weight \eqref{eq:InitWeightPowerIkenami} and the reinforcement parameter $\Delta=0$ (i.e. unreinforced).
\begin{itemize}
\item[(i)] If $\alpha < -1$, then for any $\varepsilon>0$,
\begin{align*}
\limsup_{n \to \infty} \dfrac{X_n}{n^{1/(1-\alpha)}}  >0, 
\quad \mbox{and} \quad 
\lim_{n \to \infty} \dfrac{X_n}{\{ n(\log n)^{1+\varepsilon}\}^{1/(1-\alpha)}} = 0 
\quad \mbox{a.s..}
\end{align*}
\item[(ii)] If $\alpha = -1$, then for any $\varepsilon>0$,
\begin{align*}
\limsup_{n \to \infty} \dfrac{X_n}{n^{(1-\varepsilon)/2}} =+\infty, 
\quad \mbox{and} \quad 
\lim_{n \to \infty}\dfrac{X_n}{\{ n(\log n)^{1+\varepsilon}\}^{1/2}} = 0
 \quad \mbox{a.s..}
\end{align*}
\item[(iii)] If $-1<\alpha \leq 1$, then for any $\varepsilon>0$,
\begin{align*}
\limsup_{n \to \infty} \dfrac{X_n}{n^{1/2}} >0, 
\quad \mbox{and} \quad 
\lim_{n \to \infty} \dfrac{X_n}{\{ n(\log n)^{1+\varepsilon}\}^{1/2}} = 0 
\quad \mbox{a.s..}
\end{align*}
\end{itemize}
\end{theorem}

\subsection{Related works} 
We briefly review related literatures concerning limit theorems for ERRWs in one dimension. In Davis \cite{Davis90}, the strong law of large numbers
\[ \lim_{n \to \infty} \dfrac{X_n}{n}=0\quad \mbox{a.s.} \]
is proved for initially fair, sequence-type RRWs (that is, $\mathbf{f}_x$ does not depend on $x$). See also Takeshima \cite{Takeshima00} for a possible generalization.
For limit theorems for sublinear ERRWs, see Davis \cite{Davis96} and T\'{o}th \cite{Toth96AOP,Toth97SSMHungary}, among others.
The continuous time vertex-reinforced jump process (VRJP) was introduced by Davis and Volkov \cite{DavisVolkov02}. The LERRW and the VRJP are known to be closely related, see Sabot and Tarr\`{e}s \cite{SabotTarres15} and references therein. The analog of Theorem \ref{thm:Takeshima00linear} for VRJP on $\mathbb{Z}_+$ is proved in Davis and Dean \cite{DavisDean10}. For the VRJP $\{X_t\}$ on $\mathbb{Z}_+$ corresponding to the LERRW with $f(x,\ell) = 1+\ell$, Davis and Volkov \cite{DavisVolkov02} shows that
\begin{align*}
\lim_{t \to \infty}  \dfrac{1}{\log t} \left(\max_{0 \leq s \leq t}X_s \right) = 2.768\cdots \quad \mbox{a.s..}
\end{align*}
In Lupu, Sabot, and Tarr\'{e}s \cite{LupuSabotTarres19}, the continuous space limit of the VRJP in one dimension is constructed, and it is also obtained as a fine mesh limit of the LERRW.

\section{Preliminaries} \label{sec:prelim} 

\subsection{Reduction of LERRW to RWRE}
Following Pemantle \cite{Pemantle88}, we introduce a random walk in random environment (RWRE), which is equivalent to the LERRW on $\mathbb{Z}_+$ with $\Delta>0$. 

Let $p_0=1$. Assume that $\{ p_i(\omega) \}_{i \in \mathbb{N} }$ is a sequence of independent random variables, and the distribution of $p_i$ is $\mbox{\rm Beta}\left(\dfrac{w_0(i)}{2\Delta},\dfrac{w_0(i-1)+\Delta}{2\Delta} \right)$, 
that is, for $0\leq \alpha < \beta \leq 1$,
\begin{align*}
 &\mathbb{P}(\alpha \leq p_i \leq \beta) \\
 &= B \left(\frac{w_0(i)}{2\Delta},\frac{w_0(i-1)+\Delta}{2\Delta} \right)^{-1} \int_{\alpha}^{\beta} u^{\frac{w_0(i)}{2\Delta}-1}(1-u)^{\frac{w_0(i-1)+\Delta}{2\Delta}-1} \,du,
\end{align*}
where
\begin{align*}
B \left(\frac{w_0(i)}{2\Delta},\frac{w_0(i-1)+\Delta}{2\Delta} \right) = \int_0^1 t^{\frac{w_0(i)}{2\Delta}-1}(1-t)^{\frac{w_0(i-1)+\Delta}{2\Delta}-1} \,dt.
\end{align*}
The expectation and variance under $\mathbb{P}$ are denoted by $\mathbb{E}[\,\cdot \,]$ and $\mathbb{V}[\,\cdot \,]$, respectively.

Given a random environment $\{p_i(\omega)\}_{i \in \mathbb{Z}_+}$, a Markov chain $\boldsymbol{Y}=\{Y_n\}$ on $\mathbb{Z}_+$ is defined by $\mathbf{P}^{\omega}_{i_0} (Y_0=i_0)=1$ and
\begin{align*}
 \begin{cases}
 \mathbf{P}^{\omega}_{i_0} (Y_{n+1}=i+1\,|\,Y_n=i)=p_i(\omega), \\
 \mathbf{P}^{\omega}_{i_0} (Y_{n+1}=i-1\,|\,Y_n=i)=q_i(\omega):=1-p_i(\omega) \\
 \end{cases} 
\end{align*}
for $n \geq 0$ and $i \in \mathbb{Z}_+$. The next result is found in \cite{Pemantle88}, Section 3. (See also Eckhoff and Rolles \cite{EckhoffRolles09} for the uniqueness of representation.)

\begin{lemma}
For any $n \geq 0$ and any $i_0,i_1,\cdots,i_n \in \mathbb{Z}_+$, we have
\begin{align*}
 P(X_1=i_1,\ldots,X_n=i_n\mid X_0=i_0) 
 = \mathbb{E}\left[ \mathbf{P}^{\omega}_{i_0}(Y_1=i_1,\ldots,Y_n=i_n) \right].
\end{align*}
\end{lemma}

\subsection{RW in a fixed environment}

In this subsection, we fix an environment $\{p_i\}$.

Define $\{\gamma_x \}_{x \in \mathbb{Z}_+}$ by
\[ \gamma_0:=1,\quad \mbox{and} \quad \gamma_x := \prod_{i=1}^{x} \dfrac{q_i}{p_i}  \quad \mbox{for $x\in \mathbb{N}$.} \]
In the electric network interpretation (see e.g. Chapter 2 in Lyons and Peres \cite{LyonsPeres16}), $\gamma_x$ is the resistance of the edge $\{x,x+1\}$.
Let
\begin{align*}
 h(x) := \sum_{i=0}^{x-1} \gamma_i  
\end{align*}
be a harmonic function with $h(0)=0$ and $h(1)=1$. The effective resistance from the origin to infinity is $\displaystyle h(\infty) = \sum_{i=0}^\infty \gamma_i$.

Using the conductance $w_x:=1/\gamma_x$ of the edge $\{x,x+1\}$, we have
\begin{align}
 p_x = \dfrac{w_x}{w_{x-1}+w_x},
 \quad \mbox{and} \quad
 q_x = \dfrac{w_{x-1}}{w_{x-1}+w_x}\quad \mbox{for $x \in \mathbb{Z}_+$},
 \label{eq:1dRWREpxqxweight}
\end{align}
where $w_{-1}:=0$. Define $\{\pi_x \}_{x \in \mathbb{Z}_+}$ by
\[ 
 \pi_x:=w_{x-1}+w_x 
 \quad \mbox{for $x\in \mathbb{Z}_+$.} \]
From \eqref{eq:1dRWREpxqxweight}, we can see that $\{\pi_x \}$ is a reversible measure. 
Notice that
\begin{align} Z:=\sum_{i=0}^{\infty} \pi_i<+\infty 
\quad \mbox{if and only if} \quad 
\sum_{i=0}^{\infty} \dfrac{1}{\gamma_i}<+\infty.
\label{equiv:PosRecCriterion}
\end{align}

The following recurrence classification is classical (see e.g. Theorem 2.2.5 in \cite{MenshikovPopovWade16}).

\begin{lemma} Consider the random walk $\boldsymbol{Y}=\{Y_n\}$ in a fixed environment $\{p_i\}$.
\begin{itemize}
\item[(i)] If $\displaystyle \sum_{x=0}^{\infty} \gamma_x <+\infty$, then $\boldsymbol{Y}$ is transient.
\item[(ii)] If $\displaystyle \sum_{x=0}^{\infty} \gamma_x = \sum_{x=0}^{\infty} \dfrac{1}{\gamma_x} =+\infty$, then $\boldsymbol{Y}$ is null recurrent.
\item[(iii)] if $\displaystyle \sum_{x=0}^{\infty} \gamma_x =+\infty$ and $\displaystyle \sum_{x=0}^{\infty} \dfrac{1}{\gamma_x} <+\infty$, then $\boldsymbol{Y}$ is positive recurrent. The unique stationary distribution is given by
\[ \pi(x):= \dfrac{1}{Z} \pi_x \quad \mbox{for $x \in \mathbb{Z}_+$.}\]
\end{itemize}
\end{lemma}

\section{Almost sure bound} \label{sec:asbound}

\subsection{Almost sure bound by the Lyapunov function method}
 
We consider the RWRE $\boldsymbol{Y}$, defined in the previous section. The first hitting time to $x \in \mathbb{Z}_+$ is defined by
\[ \tau_x := \inf \{ n \geq 0 : Y_n= x \}. \]
The next lemma is a consequence of the hitting time identity (see Proposition 2.20 in \cite{LyonsPeres16}).
\begin{lemma} \label{lem:Etaux} Define
\begin{align}
T^{\omega}(x):= \sum_{j=0}^{x-1} \pi_j \{ h(x)-h(j) \} 
=  \sum_{j=0}^{x-1} \pi_j \sum_{i=j}^{x-1} \gamma_i
=\sum_{i=0}^{x-1} \gamma_i \sum_{j=0}^i \pi_j 
\label{eq:BDchainT(x)Def}
\end{align}
for $x \in \mathbb{Z}_+$. Then the expectation of $\tau_x$ under $\mathbf{P}^{\omega}_0$ is given by $\mathbf{E}^{\omega}_0[\tau_x]=T^{\omega}(x)$.
\end{lemma}

To obtain the almost sure upper bound, we use the following lemma (see Lemma 6.1.4 and Theorem 2.8.1 in \cite{MenshikovPopovWade16}).

\begin{lemma}
\label{lem:MenshikovWade08Lem9(i)} 
Let $t_1$ be an increasing, nonnegative function on $\mathbb{Z}_+$ with $t_1(x) \to \infty$ as $x \to \infty$.
If
\[ \mbox{$\mathbb{P}$-a.e. $\omega$,} \quad T^{\omega}(x) \geq t_1(x) \quad \mbox{for all but finitely many $x \in \mathbb{Z}_+$},  \]
then for any $\varepsilon > 0$, $\mathbb{P}$-a.e. $\omega$ and $\mathbf{P}^{\omega}_0$-a.s.,
\[  Y_n \leq t_1^{-1} \bigl( 2n\{\log (2n)\}^{1+\varepsilon} \bigr)\quad \mbox{for all but finitely many $n$}. \]
\end{lemma}

As for the almost sure lower bound, we use the following version of Lemma 4.3 in \cite{HrynivMenshikovWade13}. No essential change is needed for the proof.

\begin{lemma}
\label{lem:HrynivMenshikovWade13Lem4.2bis} 
Let $t_2$ be an increasing, nonnegative function on $\mathbb{Z}_+$ with 
\[\sum_{x=1}^{\infty} \dfrac{t_2(x)}{t_2(x^2)}<\infty.
\]
If 
\begin{align*}
\mbox{$\mathbb{P}$-a.e. $\omega$,} \quad T^{\omega}(x) \leq t_2(x) \quad \mbox{for infinitely many $x \in \mathbb{Z}_+$}, 
\end{align*}
then for any $\varepsilon > 0$,
\[ \mbox{$\mathbb{P}$-a.e. $\omega$ and $\mathbf{P}^{\omega}_0$-a.s.,}\quad Y_n \geq t_2^{-1} \bigl( (1-\varepsilon)n \bigr)\quad \mbox{for all but finitely many $n$}. \]
\end{lemma}


We list useful bounds for $T^{\omega}(x)$, easily derived from \eqref{eq:BDchainT(x)Def}: Eqs. \eqref{ineq:T(x)lower} and \eqref{ineq:T(x)upper} are due to \cite{HrynivMenshikovWade13}, Lemma 3.5.

\begin{lemma}
\label{lem:HrynivMenshikovWade13Lem3.5}
For any $x \in \mathbb{Z}_+$, we have
\begin{align}
T^{\omega}(x) & \geq h(x) \geq \max_{0 \leq i < x} \gamma_i \geq \gamma_{x-1},\quad \mbox{and} \label{ineq:T(x)lower} \\
T^{\omega}(x) &\leq 2 x^2 \left( \max_{0 \leq i < x} \gamma_i\right)\left( \max_{0 \leq j < x} \dfrac{1}{\gamma_j}\right). \label{ineq:T(x)upper}
\end{align}
If $\displaystyle Z^{\omega}=\sum_{i=0}^{\infty} \pi_i<\infty$, then \eqref{ineq:T(x)upper} can be improved as follows:
\begin{align}
T^{\omega}(x) & \leq Z^{\omega} h(x) \leq Z^{\omega} x \left( \max_{0 \leq i < x} \gamma_i \right). \label{ineq:T(x)upperPosRec}
\end{align}
\end{lemma}

\subsection{LERRW with $\Delta=0$} As a warm-up, we prove almost sure bounds for the case $\Delta=0$. We use the next lemma, which is an infinite series version of l'H\^{o}pital's rule, due to Stolz and Ces\`aro. 

\begin{lemma}[see e.g. Knopp \cite{Knopp56Dover}, p. 34] \label{lem:StolzCesaro} If a real sequence $\{a_n\}$ and a positive sequence $\{b_n\}$ satisfy
\[ \lim_{n\to \infty} \dfrac{a_n}{b_n} = L \in \mathbb{R} \cup \{ \pm \infty \},\quad\mbox{and}\quad \sum_{n=1}^{\infty} b_n = +\infty, \]
then
\[ \lim_{n\to \infty} \dfrac{\sum_{k=1}^n a_k}{\sum_{k=1}^n b_k} = L. \]
\end{lemma}

\begin{proof}[Proof of Theorem \ref{thm:Takei20Main3}] Assume that $\alpha \leq 1$, and consider the random walk $\boldsymbol{Y}$ in a fixed environment $\{p_i\}$ given by \eqref{eq:1dRWREpxqxweight} and   $w_x=x^{\alpha} \vee 1$.
Notice that $\displaystyle Z=\sum_{j=0}^{\infty} \pi_j<+\infty$ if and only if $\alpha <-1$. In this subsection we write $T(x)$ for $T^{\omega}(x)$.

\noindent (i) Suppose that $\alpha<-1$. Since
\begin{align*}
h(x)=1+\sum_{i=1}^{x-1} i^{-\alpha} \sim \dfrac{1}{1-\alpha}x^{1-\alpha} \quad \mbox{as $x \to \infty$,} 
\end{align*}
it follows from Lemma \ref{lem:HrynivMenshikovWade13Lem3.5} that
for any $\varepsilon > 0$,
\begin{align*}
\dfrac{1-\varepsilon}{1-\alpha}x^{1-\alpha} \leq T(x) \leq  \dfrac{(1+\varepsilon)Z}{1-\alpha}x^{1-\alpha} \quad \mbox{for all but finitely many $x$.}
\end{align*}
By Lemmata \ref{lem:MenshikovWade08Lem9(i)} and \ref{lem:HrynivMenshikovWade13Lem4.2bis}, we have
\begin{align*}
Y_n \leq  \left(\dfrac{1-\alpha}{1-\varepsilon} \cdot 2n\{\log(2n)\}^{1+\varepsilon/2}\right)^{1/(1-\alpha)} \quad \mbox{for all large $n$,}
\intertext{and}
 Y_n \geq \left\{\dfrac{1-\alpha}{(1+\varepsilon)Z} \cdot (1-\varepsilon)n\right\}^{1/(1-\alpha)}\mbox{for infinitely many $n$,}
\end{align*}
$\mathbf{P}_0$-a.s. (Notice that $Z \geq \pi_0=1$). Thus we obtain the conclusion of (i).

\noindent (ii) When $\alpha=-1$, we have
\begin{align*}
 \gamma_i \sum_{j=0}^i \pi_j &= i\left( 2+ 2\sum_{j=1}^i \dfrac{1}{j} -\dfrac{1}{i}\right) \sim 2i \log i \quad \mbox{as $i \to \infty$.}
\end{align*}
By Lemma \ref{lem:StolzCesaro},
\begin{align*}
T(x) &= \sum_{i=0}^{x-1} \gamma_i \sum_{j=0}^i \pi_j \sim x^2\log x \quad \mbox{as $x \to \infty$.}
\end{align*}
For simplicity, we content ourselves with a weaker bound: For any $\varepsilon > 0$,
\begin{align*}
(1-\varepsilon) x^2 \leq T(x) \leq (1+\varepsilon) x^{2+\varepsilon}\quad \mbox{for all but finitely many $x$.}
\end{align*}
We can obtain the conclusion of (ii) by a similar calculation as in (i).

\noindent (iii) Suppose that $-1 < \alpha \leq 1$. We have
\begin{align*}
 \gamma_i \sum_{j=0}^i \pi_j &= i^{-\alpha}\left( 2+ 2\sum_{j=1}^i j^{\alpha} - i^{\alpha} \right) \sim 
 \dfrac{2}{\alpha+1}i \quad \mbox{as $i \to \infty$,} 
\intertext{and}
T(x) &= \sum_{i=0}^{x-1} \gamma_i \sum_{j=0}^i \pi_j \sim \dfrac{1}{\alpha+1} x^2 \quad \mbox{as $x \to \infty$,}
\end{align*}
again by Lemma \ref{lem:StolzCesaro}. The rest of the proof is the same as above.
\end{proof}

\section{Proof of main theorems} \label{sec:proofmain}

The following proposition allows us to estimate the random resistance $\{\gamma_x\}_{x\in\mathbb{Z}_+}$.

\begin{proposition} \label{prop:Takei20SxAsymp} Assume that $\{ p_i(\omega) \}_{i \in \mathbb{N} }$ is a sequence of independent random variables, and the distribution of $p_i$ is $\mbox{\rm Beta}\left(\dfrac{w_0(i)}{2\Delta},\dfrac{w_0(i-1)+\Delta}{2\Delta} \right)$. Let 
\[ S_x := \log \gamma_x = \sum_{i=1}^x \log \dfrac{1-p_i}{p_i} \quad \mbox{for $x \in \mathbb{N}$}. \]
\begin{itemize}
\item[(i)] If $\alpha<1$ and $\Delta >0$, then
\[  \lim_{x \to \infty} \dfrac{S_x}{ x^{1-\alpha}}=\dfrac{1}{K(\alpha,\Delta)}\quad \mbox{$\mathbb{P}$-a.e. $\omega$,} \]
where $K(\alpha,\Delta)$ is defined in \eqref{eq:Takei20Main1Const}. 
\item[(ii)] If $\alpha=1$ and $\Delta>0$, then
\[ \lim_{x \to \infty} \dfrac{S_x}{ \log x}=\Delta-1 \quad  \mbox{$\mathbb{P}$-a.e. $\omega$.}  \]
\end{itemize}
\end{proposition}

The proof of Proposition \ref{prop:Takei20SxAsymp} consists of several steps, and will be given in the next section.
We prove our main results first. Notice that 
\begin{align}
 \displaystyle Z^{\omega} = \sum_{x=0}^{\infty} \pi_x<+\infty \quad \mbox{if $\alpha<1$ and $\Delta>0$, or $\alpha=1$ and $\Delta>2$.} \label{cond:LERRWPosRec}
\end{align}

\begin{proof}[Proof of Theorem \ref{thm:Takei20Main1}] We fix $\alpha<1$ and $\Delta>0$, and write $K=K(\alpha,\Delta)$. 

By Proposition \ref{prop:Takei20SxAsymp} (i) and Eq. \eqref{ineq:T(x)lower}, for any $\varepsilon>0$, 
\begin{align*}
 T^{\omega}(x) \geq \gamma_{x-1} \geq \exp\left(\dfrac{1-\varepsilon}{K} x^{1-\alpha} \right) \quad \mbox{for all large $x$.}
\end{align*}
By Lemma \ref{lem:MenshikovWade08Lem9(i)}, $\mathbb{P}$-a.e. $\omega$ and $\mathbf{P}^{\omega}_0$-a.s.,
\begin{align*}
Y_n \leq \left\{ \dfrac{K \log  [2n\log \{(2n)^{1+\varepsilon} \}]}{1-\varepsilon}\right\}^{1/(1-\alpha)} \quad \mbox{for all large $n$,}
\end{align*}
which implies 
\begin{align*}
\limsup_{n \to \infty} \dfrac{Y_n}{(K\log n)^{1/(1-\alpha)}} \leq \dfrac{1}{(1-\varepsilon)^{1/(1-\alpha)}}. 
\end{align*}
Thus we have
\begin{align*}
\limsup_{n \to \infty} \dfrac{X_n}{(K\log n)^{1/(1-\alpha)}} \leq  1\quad \mbox{$P$-a.s..}
\end{align*}

We turn to the lower bound. Fix an arbitrary $\varepsilon>0$. Using Proposition \ref{prop:Takei20SxAsymp} (i), we can see that
\begin{align*}
\max_{0 \leq i < x} \gamma_i \leq \exp\left( \dfrac{1+\varepsilon/2}{K} x^{1-\alpha} \right)\quad \mbox{for all large $x$.}
\end{align*}
By \eqref{cond:LERRWPosRec},
\begin{align*}
Z^{\omega} x \leq \exp\left( \dfrac{\varepsilon/2}{K} x^{1-\alpha} \right) \quad \mbox{for all large $x$.}
\end{align*}
It follows from \eqref{ineq:T(x)upperPosRec} that
\begin{align*}
T^{\omega}(x) \leq \exp\left( \dfrac{1+\varepsilon}{K} x^{1-\alpha} \right)\quad \mbox{ for all large $x$.}
\end{align*} 
By Lemma \ref{lem:HrynivMenshikovWade13Lem4.2bis}, $\mathbb{P}$-a.e. $\omega$ and $\mathbf{P}^{\omega}_0$-a.s.,
\begin{align*}
Y_n \geq \left\{ \dfrac{K\log (1-\varepsilon)n}{1+\varepsilon}\right\}^{1/(1-\alpha)} \mbox{for infinitely many $n$,} 
\end{align*}
which implies
\begin{align*}
\limsup_{n \to \infty} \dfrac{Y_n}{(K\log n)^{1/(1-\alpha)}} \geq \dfrac{1}{(1+\varepsilon)^{1/(1-\alpha)}}. 
\end{align*}
Thus we have
\begin{align*}
\limsup_{n \to \infty} \dfrac{X_n}{(K\log n)^{1/(1-\alpha)}} \geq  1\quad \mbox{$P$-a.s..}
\end{align*}
This completes the proof.
\end{proof}

\begin{proof}[Proof of Theorem \ref{thm:Takei20Main2}] The proof of the case (i) (resp. (ii)) is closely related to that of the case (i) (resp. (iii)) of Theorem \ref{thm:Takei20Main3}.

\noindent (i) Assume that $\alpha=1$ and $\Delta>2$. Proposition \ref{prop:Takei20SxAsymp} (ii) implies that for any $\varepsilon \in (0,1)$,
\begin{align*}
 \gamma_i \geq \exp((\Delta-1-\varepsilon)\log i)=i^{\Delta-1-\varepsilon}\quad \mbox{for all large $i$}.
\end{align*}
By \eqref{ineq:T(x)lower},
\begin{align}
T^{\omega}(x) \geq h(x)=\sum_{i=0}^{x-1} \gamma_i \geq  \dfrac{(1-\varepsilon)x^{\Delta-\varepsilon}}{\Delta-\varepsilon} \quad \mbox{for all large $x$.} \label{ineq:Takei20Main2T(x)lower1}
\end{align}
By Lemma \ref{lem:MenshikovWade08Lem9(i)}, $\mathbb{P}$-a.e. $\omega$ and $\mathbf{P}^{\omega}_0$-a.s.,
\begin{align*}
Y_n \leq  \left(\dfrac{\Delta-\varepsilon}{1-\varepsilon} \cdot 2n\{ \log (2n) \}^{1+\varepsilon} \right)^{1/(\Delta-\varepsilon)} \quad \mbox{for all large $n$,}
\end{align*}
which implies that
\begin{align*}
\lim_{n \to \infty} \dfrac{X_n}{n^{(1+\varepsilon)/(\Delta-\varepsilon)}} =0
 \quad \mbox{$P$-a.s..}
\end{align*}
Now we turn to the lower bound. By Proposition \ref{prop:Takei20SxAsymp} (ii) and \eqref{cond:LERRWPosRec}, for any $\varepsilon>0$, 
\begin{align*}
 \gamma_i \leq \exp((\Delta-1+\varepsilon/2) \log i)=i^{\Delta-1+\varepsilon/2},
\end{align*}
and $Z^{\omega} \leq i^{\varepsilon/2}$ for all large $i$, which together with \eqref{ineq:T(x)upperPosRec} imply that
\begin{align*}
T^{\omega} (x) \leq Z^{\omega} h (x) \leq  \dfrac{(1+\varepsilon)x^{\Delta+\varepsilon}}{\Delta+\varepsilon/2}\quad \mbox{for all large $x$.}
\end{align*}
By Lemma \ref{lem:HrynivMenshikovWade13Lem4.2bis}, $\mathbb{P}$-a.e. $\omega$ and $\mathbf{P}^{\omega}_0$-a.s.,
\begin{align*}
 Y_n \geq \left\{\dfrac{\Delta+\varepsilon/2}{1+\varepsilon} \cdot (1-\varepsilon)n\right\}^{1/(\Delta+\varepsilon)}\mbox{for infinitely many $n$,}
\end{align*}
which implies that
\begin{align*}
\limsup_{n \to \infty} \dfrac{X_n}{n^{1/(\Delta+2\varepsilon)}}=+\infty 
\quad \mbox{$P$-a.s..}
\end{align*}

\noindent
(ii) First we assume that $\alpha=1$ and $0<\Delta< 2$. For any $\varepsilon > 0$,
\begin{align*}
i^{\Delta-1-\varepsilon} \leq \gamma_i \leq i^{\Delta-1+\varepsilon}\quad \mbox{for all large $i$.}
\end{align*}
Notice that
\begin{align*}
T^{\omega}(x) &=\sum_{i=0}^{x-1} \gamma_i \sum_{j=0}^i \pi_j = 1+ \sum_{i=1}^{x-1} \gamma_i \left( 2 + 2 \sum_{j=1}^i \dfrac{1}{\gamma_j} - \dfrac{1}{\gamma_i} \right),
\end{align*}
and
\begin{align*}
\sum_{i=1}^{x-1} i^{\Delta-1 \pm \varepsilon} \sum_{j=1}^i j^{-(\Delta-1) \pm \varepsilon} \sim \dfrac{1}{(2-\Delta\pm \varepsilon)(2\pm2\varepsilon)}x^{2\pm2\varepsilon} \quad \mbox{as $x \to \infty$.}
\end{align*}
For any $\varepsilon>0$, we have
\begin{align}
T^{\omega}(x) \leq \dfrac{1+\varepsilon}{(2-\Delta+\varepsilon)(2+2\varepsilon)}x^{2+2\varepsilon} \quad \mbox{for all large $x$.} \label{ineq:Takei20Main2T(x)upper2}
\end{align}
On the other hand, for any $\varepsilon \in (0,(\Delta -2) \wedge 1)$,
\begin{align*}
T^{\omega}(x) \geq \dfrac{1-\varepsilon}{(2-\Delta-\varepsilon)(2-2\varepsilon)}x^{2-2\varepsilon} \quad \mbox{for all large $x$.} 
\end{align*}
The rest of proof is similar to that of Theorem \ref{thm:Takei20Main3} (iii).

For the case $\alpha=1$ and $\Delta=2$, note that \eqref{ineq:Takei20Main2T(x)lower1} is actually valid for any $\Delta>0$, and \eqref{ineq:Takei20Main2T(x)upper2} is true also for $\Delta=2$. This completes the proof.
\end{proof}

\section{Proof of Proposition \ref{prop:Takei20SxAsymp}}
\label{sec:ProofPropTakei20SxAsymp}

Define $\zeta_i:=\log \dfrac{1-p_i}{p_i}$ for $i \in \mathbb{N}$. We begin with a particularly simple case, $\alpha=0$ and $\Delta>0$.
Since $\{p_i\}_{i \in \mathbb{N}}$ is an i.i.d. sequence 
with $\mbox{Beta} \left(\dfrac{1}{2\Delta},\dfrac{1+\Delta}{2\Delta} \right)$, 
the strong law of large numbers for i.i.d. sequences together with (4.1) in \cite{Takeshima00} imply that
\[ 
\lim_{x \to \infty} \dfrac{S_x}{x} = \mathbb{E} [\zeta_1]=\Psi\left(\frac{1}{2\Delta}+\dfrac{1}{2} \right) - \Psi\left(\frac{1}{2\Delta}\right) 
\quad \mbox{$\mathbb{P}$-a.e. $\omega$.}
\]
If $\Delta=1$, then we have
\begin{align}
 \Psi(1) - \Psi\left(\frac{1}{2}\right) &= \int_0^1 \left( \log\dfrac{1-u}{u}\right) \cdot \dfrac{1}{2\sqrt{u}}\,du \notag \\
 &= \int_0^1 \{\log(1+t) + \log (1-t) - 2 \log t\} \,dt \quad (t=\sqrt{u}) \notag \\
&= \log 4.
 \label{eq:digammaDifflog4}
\end{align}

To obtain the result for the other cases, we prepare some lemmata.
From (4.15) and (4.13) in \cite{Takeshima00}, for $x \in \mathbb{N}$, 
\begin{align}
\mathbb{E}[S_x] 
&= \sum_{i=1}^x \left\{ \Psi\left( \dfrac{w_0(i-1)+\Delta}{2\Delta} \right) - \Psi\left(\dfrac{w_0(i)}{2\Delta}\right) \right\} \label{eq:Takeshima00(4.15)a} \\
&= \Psi\left( \dfrac{w_0(0)}{2\Delta} \right) - \Psi\left( \dfrac{w_0(x)}{2\Delta} \right) + \sum_{i=0}^{x-1} \left\{ \Psi\left( \dfrac{w_0(i)}{2\Delta} +\dfrac{1}{2}\right) - \Psi\left(\dfrac{w_0(i)}{2\Delta}\right) \right\}, \label{eq:Takeshima00(4.15)b}\\
\mathbb{V}[S_x] 
&= \sum_{i=1}^x \left\{ \Psi'\left( \dfrac{w_0(i-1)+\Delta}{2\Delta} \right) + \Psi'\left(\dfrac{w_0(i)}{2\Delta}\right) \right\}. \label{eq:Takeshima00(4.13)}
\end{align}

\begin{lemma} \label{lem:Takeshima00(4.10)(4.12)improved}
We have
\begin{align}
 \Psi'(z) \sim  \begin{cases}
\dfrac{1}{z} & (z \to \infty), \\[2mm]
\dfrac{1}{z^2} &(z \to 0), \\
\end{cases} 
\label{asymp:Takeshima00(4.12)improved}
\end{align}
and
\begin{align}
 \Psi\left( z +\dfrac{1}{2}\right) - \Psi\left(z\right) \sim \begin{cases}
\dfrac{1}{2z} & (z \to \infty), \\[2mm]
\dfrac{1}{z} &(z \to 0). \\
\end{cases}
\label{asymp:Takeshima00(4.10)improved}
\end{align}
\end{lemma}

\begin{proof} The series expansions of $\Psi(z)$ and $\Psi'(z)$ are given by
\begin{align}
\Psi(z) &= -\gamma-\sum_{k=0}^{\infty} \left( \dfrac{1}{z+k}-\dfrac{1}{1+k} \right), \label{eq:Takeshima00(4.6)} \\
\Psi'(z) &= \sum_{k=0}^{\infty} \dfrac{1}{(z+k)^2}, \label{eq:Takeshima00(4.7)}
\end{align}
where $\gamma$ is Euler's constant. Eq. \eqref{asymp:Takeshima00(4.12)improved} can be obtained as follows:
\begin{align*}
\Psi'(z) = \dfrac{1}{z^2}+\sum_{\ell=1}^{\infty} \dfrac{1}{(z+\ell)^2}
\begin{cases}
\displaystyle \geq \dfrac{1}{z^2} + \int_1^{\infty} \dfrac{1}{(z+x)^2}\,dx = \dfrac{1}{z^2}+\dfrac{1}{z+1}, \\[3mm]
\displaystyle \leq \dfrac{1}{z^2} +\int_0^{\infty} \dfrac{1}{(z+x)^2}\,dx = \dfrac{1}{z^2}+\dfrac{1}{z}. \\
\end{cases}
\end{align*}
We turn to \eqref{asymp:Takeshima00(4.10)improved}. Since $\Psi$ is increasing and $\Psi'$ is decreasing, 
\begin{align}
\Psi\left( z + \dfrac{1}{2} \right) - \Psi(z) \begin{cases} \leq \Psi\left( z + 1 \right) - \Psi(z) = \dfrac{1}{z}, 
\vspace{1mm} \\[3mm]
\displaystyle = \int_z^{z+1/2} \Psi'(u) \,du  \leq \dfrac{1}{2} \Psi'(z) \leq\dfrac{1}{2z}+\dfrac{1}{2z^2}. 
\end{cases}
\label{ineq:Takeshima00(4.10)bis1}
\end{align}
We use the first bound for $0<z \leq 1$, and the second for $z \geq 1$.
On the other hand, by \eqref{eq:Takeshima00(4.6)}, 
\begin{align}
 \Psi\left( z + \dfrac{1}{2} \right) - \Psi(z) 
= \dfrac{1}{2}\sum_{\ell=0}^{\infty} \dfrac{1}{(z+\ell)(z+\frac{1}{2}+\ell)} \geq \dfrac{1}{z(2z+1)},
\label{ineq:Takeshima00(4.10)bis2}
\end{align}
and from (4.10) in \cite{Takeshima00},
\begin{align}
\Psi\left( z + \dfrac{1}{2} \right) - \Psi(z) \geq \dfrac{1}{2z}.
\label{ineq:Takeshima00(4.10)bis3}
\end{align}
We use \eqref{ineq:Takeshima00(4.10)bis2} for $0<z \leq 1/2$, and \eqref{ineq:Takeshima00(4.10)bis3} for $z \geq 1/2$.
Thus \eqref{asymp:Takeshima00(4.10)improved} follows from \eqref{ineq:Takeshima00(4.10)bis1}--\eqref{ineq:Takeshima00(4.10)bis3}.
\end{proof}

We use the following bound also ((4.11) in \cite{Takeshima00}): For $s,t>0$, 
\begin{align}
 \log t - \log s - \dfrac{1}{t} \leq \Psi(t) - \Psi(s) \leq \log t - \log s + \dfrac{1}{s}. \label{eq:Takeshima00(4.11)}
\end{align}

\begin{lemma} \label{lem:E(Sx)asymp} Assume that $\alpha \leq 1$ and $\Delta>0$. As $x \to \infty$,
\begin{align*}
\mathbb{E}[S_x] \sim \begin{cases}
\dfrac{2\Delta}{1-\alpha} x^{1-\alpha}&(\alpha<0), \\
\dfrac{\Delta}{1-\alpha} x^{1-\alpha}&(0<\alpha<1), \\
(\Delta-1) \log x&(\alpha=1,\,\Delta \neq 1). \\
\end{cases}
\end{align*}
If $\alpha=\Delta=1$, then $\mathbb{E}[S_x] \equiv \log 4$ for all $x \in \mathbb{N}$, by \eqref{eq:Takeshima00(4.15)a} and \eqref{eq:digammaDifflog4}.
\end{lemma}

\begin{proof} By \eqref{eq:Takeshima00(4.11)} and
\[ \log \dfrac{w_0(0)}{2\Delta} - \log \dfrac{w_0(x)}{2\Delta} = -\alpha \log x, \]
the first term in \eqref{eq:Takeshima00(4.15)b} satisfies 
\begin{align}
 - \alpha \log x -2\Delta\leq
\Psi\left( \dfrac{w_0(0)}{2\Delta} \right) - \Psi\left( \dfrac{w_0(x)}{2\Delta} \right) \leq - \alpha \log x +2\Delta x^{-\alpha}.
\label{ineq:E(Sx)asymp}
\end{align}

First we assume that $0<\alpha \leq 1$. From \eqref{ineq:E(Sx)asymp}, we have
\begin{align}
\Psi\left( \dfrac{w_0(0)}{2\Delta} \right) - \Psi\left( \dfrac{w_0(x)}{2\Delta} \right) \sim - \alpha \log x \quad \mbox{as $x \to \infty$.}
\label{asymp:E(Sx)asympPos1}
\end{align}
Since $w_0(x) \to \infty$ as $x \to \infty$, by \eqref{asymp:Takeshima00(4.10)improved},
\begin{align*}
\Psi\left( \dfrac{w_0(i)}{2\Delta} +\dfrac{1}{2}\right) - \Psi\left(\dfrac{w_0(i)}{2\Delta}\right) \sim \dfrac{\Delta}{i^{\alpha}} \quad \mbox{as $i \to \infty$,}
\end{align*}
which implies
\begin{align}
&\sum_{i=0}^{x-1} \left\{ \Psi\left( \dfrac{w_0(i)}{2\Delta} +\dfrac{1}{2}\right) - \Psi\left(\dfrac{w_0(i)}{2\Delta}\right) \right\} \notag \\
&\sim \Delta \sum_{i=1}^{x-1}  \dfrac{1}{i^{\alpha}} 
\sim 
\begin{cases}
\dfrac{\Delta}{1-\alpha} x^{1-\alpha} &(0<\alpha<1), \\
\Delta \log x & (\alpha=1) \\
\end{cases}
\quad \mbox{as $x \to \infty$.}
\label{asymp:E(Sx)asympPos2}
\end{align}
Eqs. \eqref{asymp:E(Sx)asympPos1} and \eqref{asymp:E(Sx)asympPos2} give the conclusion for $0<\alpha \leq 1$.

Next we assume that $\alpha < 0$. Since $w_0(x)$ vanishes as $x \to \infty$, by \eqref{asymp:Takeshima00(4.10)improved},
\begin{align*}
\Psi\left( \dfrac{w_0(i)}{2\Delta} +\dfrac{1}{2}\right) - \Psi\left(\dfrac{w_0(i)}{2\Delta}\right) \sim  \dfrac{2\Delta}{i^{\alpha}}\quad \mbox{as $i \to \infty$,}
\end{align*}
which implies
\begin{align*}
\sum_{i=0}^{x-1} \left\{ \Psi\left( \dfrac{w_0(i)}{2\Delta} +\dfrac{1}{2}\right) - \Psi\left(\dfrac{w_0(i)}{2\Delta}\right) \right\} \sim 2\Delta \sum_{i=1}^{x-1}  \dfrac{1}{i^{\alpha}} \sim \dfrac{2\Delta}{1-\alpha} x^{1-\alpha}\quad \mbox{as $x \to \infty$.}
\end{align*}
In view of \eqref{ineq:E(Sx)asymp}, the second term is dominant in \eqref{eq:Takeshima00(4.15)b} as $x \to \infty$.
\end{proof}

\begin{lemma} \label{lem:V(Sx)asymp} Assume that $\alpha \leq 1$ and $\Delta>0$. As $x \to \infty$,
\begin{align*}
\mathbb{V}[S_x] \sim \begin{cases}
\dfrac{4\Delta^2}{1-2\alpha}x^{1-2\alpha}&(\alpha<0), \\[1mm]
\dfrac{4\Delta}{1-\alpha} x^{1-\alpha}&(0<\alpha<1), \\
4\Delta \log x&(\alpha=1). \\
\end{cases}
\end{align*}
\end{lemma}

\begin{proof}
First we assume that $0<\alpha \leq 1$. Since $w_0(x) \to \infty$ as $x \to \infty$, by \eqref{asymp:Takeshima00(4.12)improved},
\begin{align*}
\Psi'\left(\dfrac{w_0(i)}{2\Delta}\right) &\sim \dfrac{2\Delta}{i^{\alpha}}, \quad \mbox{and} \\
\Psi'\left( \dfrac{w_0(i-1)+\Delta}{2\Delta} \right) &\sim \dfrac{2\Delta}{(i-1)^{\alpha}+\Delta}\sim \dfrac{2\Delta}{i^{\alpha}} \quad \mbox{as $i \to \infty$.}
\end{align*}
This together with \eqref{eq:Takeshima00(4.13)} implies 
\begin{align*}
\mathbb{V}[S_x] \sim 4\Delta \sum_{i=1}^x \dfrac{1}{i^{\alpha}} \sim \begin{cases}
\dfrac{4\Delta}{1-\alpha} x^{1-\alpha} &(0<\alpha<1), \\
4\Delta \log x&(\alpha=1)
\end{cases} \quad \mbox{as $x \to \infty$}.
\end{align*}

Next we assume that $\alpha < 0$. As $w_0(x)$ vanishes as $x \to \infty$, by \eqref{asymp:Takeshima00(4.12)improved},
\begin{align*}
\Psi'\left(\dfrac{w_0(i)}{2\Delta}\right) &\sim \left(\dfrac{2\Delta}{i^{\alpha}}\right)^2 =\dfrac{4\Delta^2}{i^{2\alpha}}, \quad \mbox{and}\\
\Psi'\left( \dfrac{w_0(i-1)+\Delta}{2\Delta} \right) &\to 
\Psi'\left( \dfrac{1}{2}\right)>0
\quad \mbox{as $i \to \infty$,}
\end{align*}
which implies
\begin{align*}
\sum_{i=1}^x  \Psi'\left( \dfrac{w_0(i-1)+\Delta}{2\Delta} \right) &\sim \Psi'\left( \dfrac{1}{2}\right)x,\quad \mbox{and}\\
\sum_{i=1}^x \Psi'\left(\dfrac{w_0(i)}{2\Delta}\right) & \sim 4\Delta^2\sum_{i=1}^x\dfrac{1}{i^{2\alpha}} \sim \dfrac{4\Delta^2}{1-2\alpha}x^{1-2\alpha}
\end{align*}
as $x \to \infty$. This together with \eqref{eq:Takeshima00(4.13)} gives the conclusion.
\end{proof}

The following lemma is a consequence of Kolmogorov's strong law of large numbers (see e.g. \cite{Ito84}, Theorem 4.5.2).
\begin{lemma} \label{lem:takeshima00Lemma4.3} 
Let $\{ \zeta_i \}$ be a sequence of independent, square-integrable random variables, and 
$\displaystyle S_x:=\sum_{i=1}^x \zeta_i$.
If $V[S_x]$ diverges as $x \to \infty$, then for any $\delta >0$,
\[ \lim_{x \to \infty} \dfrac{ S_x-E[S_x]}{ (V[S_x])^{1/2+\delta} } = 0\quad \mbox{a.s..} \]
In particular, if
\[\lim_{x \to \infty} \dfrac{(V[S_x])^{1/2+\delta}}{E[S_x]} = 0 \quad \mbox{for some $\delta>0$}, \]
then we have
\[ \lim_{x \to \infty} \dfrac{S_x}{E[S_x]}=1 \quad \mbox{a.s..}\]
\end{lemma}

The conclusion of Proposition \ref{prop:Takei20SxAsymp} for $\alpha \leq 1$ and $\alpha \neq 0$ follows from Lemmata \ref{lem:E(Sx)asymp}, \ref{lem:V(Sx)asymp}, and \ref{lem:takeshima00Lemma4.3}. 

\appendix

\section{Recurrence and transience preservation}
\label{sec:AppendixRecTraPres}

In this section, we give a new proof of Theorem \ref{thm:Takeshima00linear}.
Consider the ERRW $\boldsymbol{X}$ defined in Section \ref{sec:DefResults}, and set
\begin{align*}
\Phi_x :=  \sum_{\ell=0}^{\infty} \dfrac{1}{f(\ell,x)} \quad \mbox{for $x \in \mathbb{Z}_+$.}
\end{align*}
As in \cite{Takeshima01}, the following 0--1 law is our starting point (see \cite{Sellke94} or \cite{Takeshima01} for a proof).

\begin{lemma}[Sellke's 0--1 law] \label{Sellke} Consider the ERRW $\boldsymbol{X}$ on $\mathbb{Z}_+$. If $\Phi_x  = +\infty $  for all $x \in \mathbb{Z}_+$,  then $\boldsymbol{X}$ is  either  recurrent  a.s. or transient a.s..
\end{lemma}

Let $\tau := \inf\{ n>0 : X_n = 0 \}$ and 
\[ \displaystyle M_n := \sum_{x=0}^{X_{n \wedge \tau}-1} \dfrac{1}{w_n(x)} \quad \mbox{for $n \in \mathbb{N}$}, \mbox{ with $M_0= 0$}. \]
The process $\{ M_n \}$ is in general a non-negative supermartingale, and the process $\{ \Theta_n \}$ defined by
\[ \Theta_n := M_n + \sum_{m=0}^{n-1} \left\{ \dfrac{1}{w_m (X_m )} - \dfrac{1}{w_{m+1} (X_m )} \right\} \cdot 1_{\{ X_{m \wedge \tau} < X_{(m+1) \wedge \tau} \}}, \]
is a nonnegative martingale (see Lemma 3.0 in \cite{Davis90} for details). As
\[
\Theta_{m+1}-\Theta_m = 
\begin{cases}
\dfrac{1}{w_m(X_m)} &\mbox{if $X_{m+1}=X_m+1$,} \\[3mm]
\dfrac{-1}{w_m(X_m-1)} &\mbox{if $X_{m+1}=X_m-1$} \\
\end{cases} 
\quad \mbox{for $m < \tau$,}
\]
we can rewrite
\begin{align}
\label{ineq:ERRWThetanVervoort}
\Theta_n = \sum_{x=0}^{X_n-1} \sum_{\ell=0}^{\phi_n(x)-1} \dfrac{(-1)^{\ell}}{f(\ell,x)} \qquad \mbox{for $n<\tau$.}
\end{align}
Notice that if $n<\tau$ and $0 \leq x < X_n$, then $\phi_n(x)$ is an odd integer. Eq. \eqref{ineq:ERRWThetanVervoort} is inspired by the proof of Theorem 8.2.2 in \cite{Vervoort00}, and is used in \cite{ACT19}.

Now fix $x \in \mathbb{Z}_+$, and let
\begin{align}
s_j =s_j(x):= \sum_{\ell=0}^{j-1} \dfrac{(-1)^{\ell}}{f(\ell,x)} \qquad \mbox{for $j \in \mathbb{Z}_+$.}
\label{def:ERRWThetanVervoorts_j}
\end{align}
Since $f(\ell,x)$ is non-decreasing in $\ell$, we have
\begin{align}
0 \leq s_2 \leq s_4 \leq \cdots \leq s_{2k} \leq s_{2k+2} \leq \cdots \leq s_{2k+1} \leq s_{2k-1} \leq \cdots \leq s_3 \leq s_1. 
\label{ineq:ERRWsjOrder}
\end{align}
Set $\displaystyle s_{\infty} := \lim_{k \to \infty} s_{2k}$. Note that
\begin{align}
 s_{2k} \leq s_{\infty} \leq s_{2k-1} \qquad \mbox{for $k \in \mathbb{N}$.}
 \label{ineq:ERRWsjOrder2}
\end{align}

The initial transience is always preserved for ERRWs on the half-line. 

\begin{theorem}[cf. \cite{Takeshima01}, Corollary 1.2] Consider the ERRW $\boldsymbol{X}$ on $\mathbb{Z}_+$, and assume that $\Phi_x  = +\infty $ for all $x \in \mathbb{Z}_+$. If $F_0<+\infty$, then $\boldsymbol{X}$ is transient a.s..
\end{theorem}

\begin{proof} If $n<\tau$, then by \eqref{ineq:ERRWThetanVervoort} and \eqref{ineq:ERRWsjOrder},
\begin{align*}
\Theta_n &=\sum_{x=0}^{X_n-1} s_{\phi_n(x)}(x)   \leq \sum_{x=0}^{X_n-1} s_1(x) = \sum_{x=0}^{X_n-1} \dfrac{1}{f(0,x)} \leq F_0<+\infty.
\end{align*}
Thus the martingale $\{\Theta_n\}$ is uniformly bounded, and the martingale convergence theorem shows that
\begin{align*}
E\left[ \lim_{n \to \infty} \Theta_n \right] =  \lim_{n \to \infty} E[\Theta_n]= E[\Theta_1] = \dfrac{1}{f(0,x)}>0.  
\end{align*}
This implies that $P(\tau<+\infty)<1$. By Lemma \ref{Sellke}, $\boldsymbol{X}$ is transient a.s..
\hfill
\end{proof}

The recurrence preservation is a much more delicate issue, and is explored in \cite{Davis90,Davis89,Vervoort00,Takeshima01,ACT19}. Here is the simplest case where the initial recurrence is also preserved (cf. Corollary 4.2 in \cite{Takeshima01}): If the reinforcement scheme is `up-only', that is, $\{ \mathbf{f}_x\}_{x \in \mathbb{Z}_+}$ satisfies 
\[ f(2m,x)=f(2m-1,x) \quad \mbox{for all $m \in \mathbb{N}$ and $x \in \mathbb{Z}_+$,} \]
then 
\begin{align*}
s_{2k-1} (x) &=\dfrac{1}{f(0,x)} - \sum_{m=1}^{k-1} \left(\dfrac{1}{f(2m-1,x)}-\dfrac{1}{f(2m,x)}\right) \\
&= \dfrac{1}{f(0,x)}\quad \mbox{for any $k \in \mathbb{N}$}, 
\end{align*} 
and
\begin{align*}
\Theta_n = \sum_{x=0}^{X_n-1} \dfrac{1}{f(0,x)}\quad \mbox{for $n < \tau$.}
\end{align*}
Let $E:=\left\{ \mbox{$\tau=+\infty$,\, $\displaystyle \lim_{n \to \infty} X_n=+\infty$} \right\}$, and assume that $\Phi_x  = +\infty $ for all $x \in \mathbb{Z}_+$. If $F_0=+\infty$, then $P(E)$ cannot be positive, and $\boldsymbol{X}$ is recurrent a.s., by Lemma \ref{Sellke}. 

Now we show the recurrence preservation for the LERRW.

\begin{theorem}[cf. \cite{Takeshima00}, Theorem 4.1, part 2] Consider the LERRW $\boldsymbol{X}$ on $\mathbb{Z}_+$ with the reinforcement parameter $\Delta>0$.
If $F_0=+\infty$, then $\boldsymbol{X}$ is recurrent a.s..
\end{theorem}

\begin{proof} Fix $x \in \mathbb{Z}_+$. We use $\{s_j(x)\}$ defined by \eqref{def:ERRWThetanVervoorts_j}. 
Our claim is 
\begin{align}
s_{\infty}(x) \geq \dfrac{1}{2f(0,x)+\Delta}\qquad \mbox{for each $x \in \mathbb{Z}_+$.} 
\label{ineq:1dLERRW190906main}
\end{align}
Fix an arbitrary $x \in \mathbb{Z}_+$. Since
\begin{align*}
&\dfrac{d}{dm}\left( \dfrac{1}{f(0,x)+2m\Delta} - \dfrac{1}{f(0,x)+(2m+1)\Delta}\right) \\
&= -\dfrac{2\Delta}{\{f(0,x)+2m\Delta\}^2} + \dfrac{2\Delta}{\{f(0,x)+(2m+1)\Delta\}^2} <0,
\end{align*}
we can see that 
\begin{align*}
s_{2k}(x)&= \sum_{m=0}^{k-1} \left( \dfrac{1}{f(0,x)+2m\Delta} - \dfrac{1}{f(0,x)+(2m+1)\Delta}\right) \\
&\geq \int_0^k \left( \dfrac{1}{f(0,x)+2u\Delta} - \dfrac{1}{f(0,x)+(2u+1)\Delta}\right)\,du \\
&= \left[ \dfrac{1}{2\Delta}\log \dfrac{f(0,x)+2u\Delta}{f(0,x)+(2u+1)\Delta} \right]_{u=0}^{u=k} \\ 
&= \dfrac{1}{2\Delta} \log \dfrac{\{ f(0,x)+\Delta \}\{ f(0,x)+2k\Delta \}}{f(0,x)\{ f(0,x)+(2k+1)\Delta \}} \\
&= \dfrac{1}{2\Delta} \log \dfrac{f(0,x)^2 + (2k+1)\Delta f(0,x) + 2k\Delta^2}{f(0,x)^2 + (2k+1)\Delta f(0,x)}. 
\end{align*}
Using
\begin{align*}
\dfrac{1}{2} \log \dfrac{1+t}{1-t} \geq t \qquad \mbox{for $t \in (0,1)$},
\end{align*}
we have
\begin{align*}
s_{2k}(x) &\geq \dfrac{1}{\Delta} \cdot \dfrac{k\Delta^2}{f(0,x)^2 + (2k+1)\Delta f(0,x) + k\Delta^2}\qquad \mbox{for $k \in \mathbb{N}$.} 
\end{align*}
Letting $k \to \infty$, we obtain \eqref{ineq:1dLERRW190906main}.

By \eqref{ineq:ERRWsjOrder2} and \eqref{ineq:1dLERRW190906main}, we have
\begin{align*}
\Theta_n = \sum_{x=0}^{X_n-1} s_{\phi_n(x)}(x) \geq \sum_{x=0}^{X_n-1} \dfrac{1}{2f(0,x)+\Delta}\qquad \mbox{for $n < \tau$.}
\end{align*}
If $F_0=+\infty$, then $P(E)$ cannot be positive.
The conclusion follows from Lemma \ref{Sellke}.
\hfill
\end{proof}


\begin{thebibliography}{99}
\bibitem{ACT19} 
Akahori, J., Collevecchio, A., and Takei, M. (2019). Phase transitions in edge-reinforced random walks on the half-line, {\it Electron. Commun. Probab.}, {\bf 24}, paper no. 39, 1--12.
\bibitem{CometsMenshikovPopov98}
Comets, F., Menshikov, M., and Popov, S. (1998). Lyapunov functions for random walks and strings in random environment, {\it Ann. Probab.}, {\bf 26}, 1433--1445.
\bibitem{Davis89}
Davis, B. (1989). Loss of recurrence in reinforced random walk, {\it Almost Everywhere Convergence}, 179--188, Academic Press.
\bibitem{Davis90}
Davis, B. (1990). Reinforced random walk, {\it Probab. Theory Relat. Fields} {\bf 84}, 203--229.
\bibitem{Davis96}
Davis, B. (1996). Weak limits of perturbed random walks and the equation $\displaystyle Y_t = B_t + \alpha \sup \{ Y_s : s \leq t \} + \beta \inf \{ Y_s : s \leq t \}$,  {\it Ann. Probab.}, {\bf 24}, 2007--2023.
\bibitem{DavisDean10}
Davis, B. and Dean, N. (2010). Recurrence and transience preservation for vertex reinforced jump processes, {\it Illinois J. Math.}, {\bf 54}, 869--893.
\bibitem{DavisVolkov02}
Davis, B. and Volkov, S. (2002). Continuous time vertex-reinforced jump processes, {\it Probab. Theory Relat. Fields}, {\bf 123}, 281--300.
\bibitem{Diaconis88}
Diaconis, P. (1988). Recent progress on de Finetti's notions of exchangeability, {\it Bayesian statistics}, {\bf 3}, 111--125, Oxford University Press.
\bibitem{EckhoffRolles09}
Eckhoff, M. and Rolles, S. W. W. (2009). Uniqueness of the mixing measure for a random walk in a random environment on the positive integers, {\it Electron. Commun. Probab.},  {\bf 14},  31--35.
\bibitem{HrynivMenshikovWade13}
Hryniv, O., Menshikov, M. V. and Wade, A. R. (2013). Random walk in mixed random environment without uniform ellipticity, {\it Proc. Steklov Inst. Math.}, {\bf 282}, 106--123. 
\bibitem{Ikenami01}
Ikenami, M. (2001). On random walks in random environment (in Japanese), Master thesis, Kobe University.
\bibitem{Ito84}
Ito, K. (1984). Introduction to probability theory, Cambridge University Press.
\bibitem{KeaneRolles00}
Keane, M. S. and Rolles, S. W. W. (2000). Edge-reinforced random walk on finite graphs, {\it Infinite dimensional stochastic analysis}, 217--234, Koninklijke Nederlandse Akademie van Wetenschappen.
\bibitem{Knopp56Dover}
Knopp, K. (1956). Infinite sequences and series, Dover Publications.
\bibitem{LupuSabotTarres19}
Lupu, T., Sabot, C., and Tarr\'{e}s, P. (2020). Fine mesh limit of the VRJP in dimension one and Bass--Burdzy flow, {\it Probab. Theory Relat. Fields}, {\bf 177}, 55--90.
\bibitem{LyonsPeres16} Lyons, R. and Peres, Y. (2016). Probability on trees and networks,  Cambridge University Press.
\bibitem{MenshikovWade08SPA}
Menshikov, M. V. and Wade, A. R. (2008). Logarithmic speeds for one-dimensional perturbed random walk in random environment, {\it Stoch. Proc. Appl.}, {\bf 118}, 389--416.
\bibitem{MenshikovPopovWade16}
Menshikov, M. V., Popov, S., and Wade, A. R. (2016). Non-homogeneous random walks: Lyapunov function methods for near-critical stochastic systems, {\it Cambridge Tracts in Mathematics}, {\bf 209}, Cambridge University Press.
\bibitem{Pemantle88}
Pemantle, R. (1988). Phase transition in reinforced random walk and RWRE on trees, {\it Ann. Probab.}, {\bf 16}, 1229--1241.
\bibitem{SabotTarres15}
Sabot, C. and Tarr\`{e}s, P. (2015). Edge-reinforced random walk, vertex-reinforced jump process and the supersymmetric hyperbolic sigma model, {\it J. European Math. Soc.}, {\bf 17}, 2353--2378.
\bibitem{Sellke94}
Sellke, T. (1994). Reinforced random walks on the $d$-dimensional integer lattice, {\it Technical Report, Department of Statistics, Purdue University} {\bf \#94-26}. / (2008). {\it Markov Processes Relat. Fields}, {\bf 14}, 291--308.
\bibitem{Takeshima00}
Takeshima, M. (2000). Behavior of 1-dimensional reinforced random walk, {\it Osaka J. Math.}, {\bf 37}, 355--372.
\bibitem{Takeshima01}
Takeshima, M. (2001). Estimates of hitting probabilities for a 1-dimensional reinforced random walk, {\it Osaka J. Math.}, {\bf 38}, 693--709.
\bibitem{Toth96AOP}
T\'{o}th, B. (1996). Generalized Ray-Knight theory and limit theorems for self-interacting random walks on $\mathbb{Z}$, {\it Ann. Probab.}, {\bf 24}, 1324--1367.
\bibitem{Toth97SSMHungary}
T\'{o}th, B. (1997). Limit theorems for weakly reinforced random walks on $\mathbb{Z}$, {\it Stud. Sci. Math. Hungary}, {\bf 33}, 321--337.
\bibitem{Vervoort00}
Vervoort, M. R. (2000). Games, walks and grammers: Problems I've worked on, Ph.D thesis, Universiteit van Amsterdam.
\end{thebibliography}
\end{document}